\documentclass[12pt]{amsart}

\usepackage[bookmarks=true]{hyperref}

\font\bbbld=msbm10 scaled\magstephalf

\newcommand{\bfH}{\hbox{\bbbld H}}
\newcommand{\bfR}{\hbox{\bbbld R}}

\newcommand{\va}{{\bf a}}
\newcommand{\vb}{{\bf b}}
\newcommand{\ve}{{\bf e}}
\newcommand{\la}{\langle}
\newcommand{\ra}{\rangle}
\newcommand{\vn}{{\bf n}}

\newcommand{\vx}{{\bf x}}

\newcommand{\e}{\varepsilon}
\newcommand{\goto}{\rightarrow}
\newcommand{\ol}{\overline}

\newcommand{\be}{\begin{equation}}
\newcommand{\ee}{\end{equation}}
\newcommand{\bea}{\begin{eqnarray}}
\newcommand{\eea}{\end{eqnarray}}

\newcommand{\tD}{{\tilde D}}

\newcommand{\tg}{{\tilde g}}
\newcommand{\thh}{{\tilde h}}

\newcommand{\tkappa}{{\tilde \kappa}}
\newcommand{\tnabla}{\tilde{\nabla}}

\newtheorem{theorem}{Theorem}[section]
\newtheorem{lemma}[theorem]{Lemma}

\newtheorem{corollary}[theorem]{Corollary}

\theoremstyle{definition}

\theoremstyle{remark}

\numberwithin{equation}{section}

\begin{document}
\setlength{\baselineskip}{1.2\baselineskip}

\title[Spacelike hypersurfaces of constant curvature]
{Convex Spacelike Hypersurfaces of Constant Curvature in De Sitter Space}

\author{Joel Spruck}
\address{Department of Mathematics, Johns Hopkins University,
 Baltimore, MD 21218}
\email{js@math.jhu.edu}
\author{Ling Xiao}
\address{Department of Mathematics, Johns Hopkins University,
 Baltimore, MD 21218}
\email{lxiao8@math.jhu.edu}
\thanks{Research of the first author was supported in part NSF grant DMS 0904009.}

\maketitle

\begin{abstract}
We show that for a very general and natural class of curvature functions (for example the curvature quotients
$(\sigma_n/\sigma_l)^{\frac{1}{n-l}}$) the problem of finding a complete spacelike strictly convex hypersurface 
in de Sitter space satisfying
$f(\kappa)=\sigma \in (1,\infty)$ with a prescribed compact  future asymptotic boundary $\Gamma$ at infinity
has at least one smooth solution (if $l=1$ or $l=2$ there is uniqueness). This is the exact analogue of the asymptotic plateau problem in Hyperbolic space
 and is in fact a precise dual problem. By using this duality we obtain for free the existence of strictly convex solutions to the asymptotic Plateau problem for $\sigma_l=\sigma,\,1 \leq l <n$ in both deSitter and Hyperbolic space.
\end{abstract}

\section{Introduction}
\label{gs-I}
\setcounter{equation}{0}

Let $\bfR^{n+1,1}$ be the $(n+2)$ dimensional Minkowski space, that is,  the real vector space
$\bfR^{n+2}$ endowed with the Lorentz metric
\be \label{eq1.10}
\langle u,v\rangle =u_0 v_0+\ldots +u_n v_n-u_{n+1}v_{n+1}
\ee for all $u,v \in \bfR^{n+2}$. The one sheeted hyperboloid
\[dS_{n+1}=\{p\in \bfR^{n+2}| \la p,p\ra \,=1 \}   \]
consisting of all unit spacelike vectors and equipped with the induced metric is called de Sitter space. It is a geodesically complete simply connected Lorentzian manifold with constant curvature one . De Sitter space corresponds to  a
vacuum solution of the Einstein equations with a positive cosmological constant. \\

Choose a non-zero null vector $\va\in\bfR^{n+1,1}$ in the past half of the null cone with vertex at the origin, i.e. $\la \va,\va \ra=0,\, \la \va,\ve\ra >0$ where $\ve=(0,\ldots,0,1)$. Then the open region of de Sitter space defined by
\be \label{eq1.20}
\mathcal{H}=\{p\in dS_{n+1}|\la p,\va \ra >0\}
\ee
is called the {\em steady state space}. Since the  steady state space is only half the de Sitter space, it is incomplete. Its boundary as a subset of $dS_{n+1}$ is the null hypersurface
\be \label{eq1.30}
L_0=\{p\in dS_{n+1}|\la p,\va \ra =0\}
\ee
which represents the past infinity of $\mathcal{H}$. The spacelike hypersurfaces
\be \label{eq1.40}
L_{\tau}=\{p\in dS_{n+1}|\la p,\va \ra=\tau,\,0<\tau <\infty\}
\ee
are umbilic hypersurfaces of de Sitter space with constant mean curvature one with respect to the past oriented unit normal $N_{\tau}(p)=-p+\frac1{\tau}\va$ and foliate
the steady state space. The limit boundary $L_{\infty}$ represents a spacelike future infinity for timelike and null lines of de Sitter space \cite{HE}.\\

In this paper we are interested in finding complete spacelike (i.e. the induced metric is Riemannian) strictly locally convex immersions 
$\psi: \Sigma^n \goto \mathcal{H}^{n+1}$ with constant curvature and with prescribed (compact) future asymptotic boundary $\Gamma$ . That is we want to find $\Sigma$ satisfying
\bea \label{eq1.45} 
f(\kappa[\Sigma])&=&\sigma \\
\partial \Sigma&=&\Gamma \label{eq1.46}
\eea
where $\kappa[\Sigma]=(\kappa_1, \ldots,\kappa_n)$ denote  the positive  principal curvatures of $\Sigma$ (in the induced de Sitter metric with respect to the future oriented unit normal $N$) and $\sigma>1$ is a constant. This is the exact analogue of the problem considered by Guan and Spruck  \cite{GS11} in hyperbolic space and as we shall make precise later, the two problems are essentially dual equivalent problems. Our study was motivated by the beautiful paper of Montiel \cite{Montiel1} who treated the case of mean curvature.\\

 As  in  our earlier work \cite{NS96, RS94, GS00, GSS, GS08, GS11}, we prefer to use the half space model because we find it has great advantages. Following Montiel \cite{Montiel1} we can define a half space model for the steady state space  in the following way. Define the map 
 $\phi: \bfR^{n+1,1}\setminus\{p\in \bfR^{n+1,1}| \la p,a \ra=0\}\goto \bfR^n\times \bfR=\bfR^{n+1}$ given by
\be \label{eq1.50}
\phi(p)=\frac1{\la p,\va \ra}(p-\la p,\va \ra \vb-\la p,\vb \ra \va,1)
\ee
where $\vb\in R^{n+2}$ is a null vector such that $<\va,\vb>=1$( $\vb$ is in the  future directed null cone and $\bfR^n$ stands for the orthogonal complement of the Lorentz plane spanned by $\va$ and $\vb$). Then the image of $\mathcal{H}$ by the map $\phi$  lies in the half space $\bfR^{n+1}_+=\bfR^n\times \bfR_+$ and $\phi$  restricted to $\mathcal{H}$ is a diffeomorphism to $\bfR^{n+1}_+$. Moreover for $v\in T_p\mathcal{H}^{n+1}=T_p  dS_{n+1}$,
\be \label{eq1.60}
(d\phi)_p(v)=\frac1{\la p,\va \ra}(v-\la v,\va \ra \vb-\la v,\vb \ra \va,0)-\frac{\la v, \va \ra}{\la p, \va \ra^2} (p-\la p,\va \ra \vb-\la p,\vb \ra \va,1).
\ee
It follows that
\[\la (d\phi)_p(v),(d\phi)_p(v) \ra=\frac1{\la p, \va \ra^2}\la v, v \ra~.\]
Hence the map $\phi: \mathcal{H}^{n+1} \goto \bfR^{n+1}_+=\bfR^n\times \bfR_+$ is an isometry if $\bfR^{n+1}_+$ is endowed with the Lorentz metric
\be \label{eq1.70}
g_{(x,x_{n+1})}=\frac1{x_{n+1}^2}(dx^2-dx_{n+1}^2),
\ee
which is called the  half space model for $\mathcal{H}^{n+1}$. 
It is important to note that {\em the isometry $\phi$ reverses the time orientation.}\\

Thus $\partial_\infty \mathcal{H}^{n+1}$ is naturally identified with
$\bfR^n = \bfR^n \times \{0\} \subset \bfR^{n+1}$ and (\ref{eq1.46}) may
be understood in the Euclidean sense. For convenience we say $\Sigma$ has
compact asymptotic boundary if
$\partial \Sigma \subset \partial_\infty \mathcal{H}^{n+1}$ is compact with respect
the Euclidean metric in $\bfR^n$.\\

The curvature function $f(\lambda)$ in \eqref{eq1.45} is assumed to satisfy the fundamental structure
conditions in the convex cone
\begin{equation}
\label{eq1.80}
K:= K^+_n := \big\{\lambda \in \bfR^n:
   \mbox{each component $\lambda_i > 0$}\big\}:
 \end{equation}
\be \label{eq1.85}
\mbox{$f$ is symmetric,}
\ee
\begin{equation}
\label{eq1.90}
f_i (\lambda) \equiv \frac{\partial f (\lambda)}{\partial \lambda_i} > 0
  \;\; \mbox{in $K$}, \;\; 1 \leq i \leq n,
\end{equation}
\begin{equation}
\label{eq1.100}
\mbox{$f$ is a concave function in $K$},
\end{equation}
\be \label{eq1.110}
\mbox{the dual function $f^*(\lambda)=(f(\frac1{\lambda_1},\ldots,\frac1{\lambda_n}))^{-1}$ is also concave in K},
\ee
\begin{equation}
\label{eq1.120}
 f > 0 \;\;\mbox{in $K$},
  \;\; f = 0 \;\;\mbox{on $\partial K$}
\end{equation}
 In addition, we shall assume that $f$ is normalized
\begin{equation}
\label{eq1.130}
f(1, \dots, 1) = 1,
\end{equation}
\begin{equation}
\label{eq1.140}
\mbox{ $f$ is homogeneous of degree one}
\end{equation}
and satisfies the following more technical assumption
\begin{equation}
\label{eq1.150}
\lim_{R \rightarrow + \infty}
   f (\lambda_1, \cdots, \lambda_{n-1}, \lambda_n + R)
    \geq 1 + \varepsilon_0 \;\;\;
\mbox{uniformly in $B_{\delta_0} ({\bf 1})$}
\end{equation}
for some fixed $\varepsilon_0 > 0$ and $\delta_0 > 0$,
where $B_{\delta_0} ({\bf 1})$ is the ball
of radius $\delta_0$ centered at ${\bf 1} = (1, \dots, 1) \in \bfR^n$.\\

The assumption \eqref{eq1.110} is closely related to the well-known fact \cite{Oliker}, \cite{Gerhardt} , \cite{ Sch}
that the Gauss map $\vn$ of 
a spacelike locally strictly convex hypersurface $\Sigma^n $ in de Sitter space is an embedding into  hyperbolic space $H^{n+1}$ which inverts principal curvatures. We shall formulate a precise
global version of this correspondence (see Theorem \ref{th2.1} and Corollary \ref{cor2.1} in Section  \ref{sec1.5} )
which will be important for our deliberations. For the moment note that if 
$f = (\sigma_n/\sigma_l)^{\frac{1}{n-l}}$, $0 \leq l <  n$, defined
in $K$ where $\sigma_l$ is the normalized $l$-th elementary
symmetric polynomial ($\sigma_0 =1$), then $f^*(\lambda)=  (\sigma_{n-l}(\lambda))^{\frac{1}{n-l}}$.
Also one easily computes that
\[\lim_{R \rightarrow + \infty}
   f (\lambda_1, \cdots, \lambda_{n-1}, \lambda_n + R)
  = \Big(\frac{n}l\Big)^{\frac1{n-l}}~.\]

  Since $f$ is symmetric, by (\ref{eq1.100}),
(\ref{eq1.130}) and (\ref{eq1.140}) we have
\begin{equation}
\label{eq1.160}
f (\lambda) \leq f ({\bf 1}) + \sum f_i ({\bf 1}) (\lambda_i - 1)
= \sum f_i ({\bf 1}) \lambda_i  = \frac{1}{n} \sum \lambda_i
\;\;\mbox{in $K \subset K_1$}
\end{equation}
and
\begin{equation}
\label{eq1.170}
 \sum f_i (\lambda) = f (\lambda) + \sum f_i (\lambda) (1 - \lambda_i)
\geq f ({\bf 1}) = 1 \;\;\mbox{in $K$}.
\end{equation}
Using \eqref{eq1.110}, we see that $\sum \lambda_i^2 f^*_i(\lambda)=(f^*)^2\sum f_i(\frac1{\lambda})\geq (f^*)^2$. Since $(f^*)^*=f$, it follows that
\be \label{eq1.180}
\sum \lambda_i^2 f_i \geq f^2 \;\;\mbox{in $K$}.
\ee

 In this paper all
hypersurfaces in $\mathcal{H}^{n+1}$ we consider are assumed to be
connected and orientable. If $\Sigma$ is a complete spacelike hypersurface in
$\mathcal{H}^{n+1}$ with compact asymptotic boundary at infinity, then the
normal vector field of $\Sigma$ is chosen to be the one pointing to
the unique unbounded region in $\bfR^{n+1}_+ \setminus \Sigma$, and
the (both de Sitter and Minkowski) principal curvatures of $\Sigma$
are calculated with respect to this normal vector field.\\

Because $\Sigma$ is strictly locally convex and strictly spacelike,  we are forced to take
 $\Gamma=\partial \Omega$ where
$\Omega \subset \bfR^n$ is a smooth domain and seek $\Sigma$ as  the
graph of  a ``spacelike'' function $u(x)$ over $\Omega$, i.e.
\be \label{eq1.185}
 \Sigma=\{(x,x_{n+1}): x\in \Omega,~ x_{n+1}=u(x)\},  \hspace{.1in} |\nabla u| <1,\,\mbox{in $\ol{\Omega}$}.
 \ee
  
We will compute the  first and second fundamental forms $g_{ij},\, h_{ij}$ with respect to the induced de Sitter metric
as well as $\tilde{g}_{ij},\,\tilde{h}_{ij}$ the corresponding forms in the induced Minkowski metric viewing $\Sigma$
as a graph in the Minkowski space $\bfR^{n,1}$ with unit normal $\nu$. We use
 $$X_i=e_i+u_ie_{n+1},\;\;\ \vn=u \nu= u\frac{u_ie_i+e_{n+1}}{w},$$
 where $w=\sqrt{1-|\nabla u|^2}.$ The first fundamental form $g_{ij}$ is then given by
 \be\label{eq1.190}
 g_{ij}=\langle X_i, X_j\rangle_D=\frac{1}{u^2}(\delta_{ij}-u_iu_j)=\frac{\tg_{ij}}{u^2}.
 \ee
 For computing the second fundamental form we use
\be\label{eq1.200}
\Gamma^{n+1}_{ij}=-\frac{1}{x_{n+1}}\delta_{ij},\; \Gamma^k_{in+1}=-\frac{1}{x_{n+1}}\delta_{ik}
\ee
to obtain
\be\label{eq1.210}
\nabla_{X_i}X_j=\left(-\frac{\delta_{ij}}{x_{n+1}}+u_{ij}-\frac{u_iu_j}{x_{n+1}}\right)e_{n+1}-\frac{u_je_i+u_ie_j}{x_{n+1}}.
\ee
Then
\be\label{eq1.230}
\begin{aligned}
h_{ij}&=\langle\nabla_{X_i}X_j, u\nu\rangle_D=\frac{1}{uw}\left(\frac{\delta_{ij}}{u}-u_{ij}+\frac{u_iu_j}{u}-2\frac{u_iu_j}{u}\right)\\
&=\frac{1}{u^2w}(\delta_{ij}-u_iu_j-uu_{ij})=\frac{\thh_{ij}}{u}+\frac{\tg_{ij}}{u^2w}.
\end{aligned}
\ee
 The principal curvature $\kappa_i$ of $\Sigma$ in de Sitter space are the roots of the characteristic equation
 $$\det(h_{ij}-\kappa g_{ij})=u^{-n}\det\left(\thh_{ij}-\frac{1}{u}\left(\kappa-\frac{1}{w}\right)\tg_{ij}\right)=0.$$
Therefore,
\begin{equation}
\label{eq1.240}
\kappa_i = u \tilde{\kappa}_i + \frac{1}{w}, \;\;\; i = 1, \cdots, n.
\end{equation}

Note that from \eqref{eq1.230}, $\Sigma$ is locally strictly convex if and only if 
\be \label{eq1.245}
\mbox{ $x^2-u^2 $ is (Euclidean) locally strictly convex~.}
\ee
As in our earlier work,  we write the Minkowski principal curvatures $\tkappa[\Sigma]$ as the eigenvalues of the symmetric matrix $\tilde{A}[u]=\{\tilde{a}_{ij}\}:$
\be\label{eq1.246}
\tilde{a}_{ij}:=-\frac{1}{w}\gamma^{ik}u_{kl}\gamma^{lj}
\ee
where
\be\label{eq1.247}
\gamma^{ij}=\delta_{ij}+\frac{u_iu_j}{w(1+w)}.
\ee
By \eqref{eq1.240} the de Sitter principal curvatures $\kappa[u]$ of $\Sigma$ are the eigenvalues of the symmetric matrix $A[u]=\{a_{ij}[u]\}:$
\be\label{eq1.248}
a_{ij}[u]:=\frac{1}{w}\left(\delta_{ij}-u\gamma^{ik}u_{kl}\gamma^{lj}\right).
\ee

Define
\be \label{eq1.249}
 F(A):=f(\kappa[A])  \hspace{.1in} \mbox{ and \,\,$G(D^2u, Du, u):=F(A[u])$  }
 \ee
 where
 $A[u]=\{a_{ij}[u]\}$ is given by \eqref{eq1.248}. 
Problem (\ref{eq1.45})-(\ref{eq1.46}) then reduces to a
Dirichlet problem for a fully nonlinear second order equation 
\begin{equation}
\label{eq1.250}
G(D^2u, Du, u) = \sigma>1,
\;\; u > 0 \;\;\; \text{in $\Omega \subset \bfR^n$}
\end{equation}
with the boundary condition
\begin{equation}
\label{eq1.260}
             u = 0 \;\;\;    \text{on $\partial \Omega$.}
\end{equation}

We seek solutions of equation (\ref{eq1.250}) satisfying the spacelike condition \eqref{eq1.85} and 
(\ref{eq1.245}). Following the literature we call such solutions
{\em admissible}. By \cite{CNS3} condition~(\ref{eq1.245}) implies
that equation (\ref{eq1.250}) is elliptic for admissible solutions.
Our goal is to show that the Dirichlet problem
(\ref{eq1.250})-(\ref{eq1.260}) admits smooth admissible solutions
for all $ \sigma>1$ which is optimal.

Our main result of the paper may be stated as follows.

\begin{theorem}
\label{th1}
Let $\Gamma = \partial \Omega \times \{0\} \subset \partial_{\infty}\mathcal{H}^{n+1}$
where $\Omega$ is a bounded smooth domain in $\bfR^n$.
Suppose that $\sigma >1$  and that $f$ satisfies
 conditions (\ref{eq1.85})-(\ref{eq1.150}) with $K = K_n^+$.
Then there exists a complete locally strictly convex spacelike hypersurface
$\Sigma$ in $\mathcal{H}^{n+1}$ satisfying (\ref{eq1.45})-(\ref{eq1.46})
with uniformly bounded principal curvatures
\begin{equation}
\label{eq1.270}
 |\kappa[\Sigma]| \leq C \;\; \mbox{on $\Sigma$}.
\end{equation}
Moreover, $\Sigma$ is the graph of an admissible solution $u \in C^\infty
(\Omega) \cap C^1 (\bar{\Omega})$ of the Dirichlet problem
(\ref{eq1.250})-(\ref{eq1.260}). Furthermore, $u^2 \in C^{\infty} (\Omega)\cap C^{1,1}(\ol{\Omega}) $
and
\begin{equation}
\label{eq1.280}
\begin{aligned}
&\, u|D^2 u| \leq C
\;\;\; \mbox{in $\Omega$}, \\
&\, \sqrt{1 -|Du|^2} = \frac{1}{\sigma}
\;\;\; \mbox{on $\partial \Omega$}
\end{aligned}
\end{equation}
\end{theorem}

As a concrete application we have existence for the canonical curvature functions.
\begin{corollary}\label{cor1.1}Let $\Gamma = \partial \Omega \times \{0\} \subset \partial_{\infty}\mathcal{H}^{n+1}$
where $\Omega$ is a bounded smooth domain in $\bfR^n$.
Then there exists a complete locally strictly convex spacelike hypersurface
$\Sigma$ in $\mathcal{H}^{n+1}$ satisfying 
\[ (\sigma_n/\sigma_l)^{\frac{1}{n-l}}=\sigma>1, \hspace{.1in}0 \leq l <  n\]
with $\partial \Sigma=\Gamma$ and uniformly bounded principal curvatures
\begin{equation}
 |\kappa[\Sigma]| \leq C \;\; \mbox{on $\Sigma$}.
\end{equation}
Moreover, $\Sigma$ is the graph of an admissible solution $u \in C^\infty
(\Omega) \cap C^1 (\bar{\Omega})$ of the Dirichlet problem
(\ref{eq1.250})-(\ref{eq1.260}). Furthermore, $u^2 \in C^{\infty} (\Omega)\cap C^{1,1}(\ol{\Omega}) $
and
\begin{equation}
\begin{aligned}
&\, u|D^2 u| \leq C
\;\;\; \mbox{in $\Omega$}, \\
&\, \sqrt{1 -|Du|^2} = \frac{1}{\sigma}
\;\;\; \mbox{on $\partial \Omega$}
\end{aligned}
\end{equation}
\end{corollary}

As we mentioned earlier, in Section  \ref{sec1.5}  we prove strong duality theorems
(see Theorem \ref{th2.1} and Corollary \ref{cor2.1}  ) which allows us to transfer our existence results for $\bfH^{n+1}$
in \cite{GS11} to $\mathcal{H}^{n+1}$ and conversely. In particular we have

\begin{corollary}\label{cor1.2}Let $\Gamma = \partial \Omega \times \{0\} \subset \partial_{\infty}\mathcal{H}^{n+1}$
where $\Omega$ is a bounded smooth domain in $\bfR^n$.
Then there exists a complete locally strictly convex spacelike hypersurface
$\Sigma$ in $\mathcal{H}^{n+1}$ satisfying 
\[ (\sigma_l)^{\frac1l}=\sigma>1, \hspace{.1in} 1 \leq l \leq  n\] 
with $\partial \Sigma=\Gamma$ and uniformly bounded principal curvatures
\begin{equation}
 |\kappa[\Sigma]| \leq C \;\; \mbox{on $\Sigma$}.
\end{equation}
Moreover, $\Sigma$ is the graph of an admissible solution $u \in C^\infty
(\Omega) \cap C^1 (\bar{\Omega})$ of the Dirichlet problem
(\ref{eq1.250})-(\ref{eq1.260}). Furthermore, $u^2 \in C^{\infty} (\Omega)\cap C^{1,1}(\ol{\Omega}) $
and
\begin{equation}
\begin{aligned}
&\, u|D^2 u| \leq C
\;\;\; \mbox{in $\Omega$}, \\
&\, \sqrt{1 -|Du|^2} = \frac{1}{\sigma}
\;\;\; \mbox{on $\partial \Omega$}
\end{aligned}
\end{equation}
Further, if  $l=1$ or $l=2$ (mean curvature and normalized scalar curvature) we have uniqueness among  convex solutions and even among all solutions convex or not
if {\em $\Omega$ is simple.} 
\end{corollary}

The uniqueness part of Corollary \ref{cor1.2} follows from the uniqueness Theorem 1.6 of \cite{GS11} and a continuity
and deformation argument like that used in \cite{RS94}. Note that Montiel \cite{Montiel1} proved existence for 
$H=\sigma>1$ assuming $\partial \Omega$ is mean convex. Our result shows that for arbitrary $\Omega$ there is always a locally strictly convex solution. If $\Omega$ is simple and mean convex the solutions constructed by Montiel must agree with the ones we construct.\\

Transferring by duality the results of Corollary \ref{cor1.1} to $\bfH^{n+1}$ gives the mildly surprising
\begin{corollary}\label{cor1.3}Let $\Gamma = \partial \Omega \times \{0\} \subset \partial_{\infty} \bfH^{n+1}$
where $\Omega$ is a bounded smooth domain in $\bfR^n$.
Then there exists a complete locally strictly convex  hypersurface
$\Sigma$ in $\bfH^{n+1}$ satisfying 
\[ (\sigma_l)^{\frac1l}=\sigma^{-1} \in (0,1), \hspace{.1in} 1 \leq l < n\] 
with $\partial \Sigma =\Gamma$ and uniformly bounded principal curvatures
\begin{equation}
 |\kappa[\Sigma]| \leq C \;\; \mbox{on $\Sigma$}.
\end{equation}
Moreover, $\Sigma$ is the graph of an admissible solution $v \in C^\infty
(\Omega) \cap C^1 (\bar{\Omega})$ of the Dirichlet problem dual to
(\ref{eq1.250})-(\ref{eq1.260}). Furthermore, $v^2 \in C^{\infty} (\Omega)\cap C^{1,1}(\ol{\Omega}) $
and
\begin{equation}
\begin{aligned}
&\, v|D^2 v| \leq C
\;\;\; \mbox{in $\Omega$}, \\
&\, \frac1{\sqrt{1 +|Dv|^2}} = \frac{1}{\sigma}
\;\;\; \mbox{on $\partial \Omega$}
\end{aligned}
\end{equation}
\end{corollary}

Equation (\ref{eq1.250}) is
singular where $u = 0$. It is therefore natural to approximate the
boundary condition (\ref{eq1.260}) by
\begin{equation}
\label{eq1.290}
             u = \epsilon > 0 \;\;\;  \text{on $\partial \Omega$}.
\end{equation}
When $\epsilon$ is sufficiently small, we shall show that the Dirichlet problem
(\ref{eq1.250}),(\ref{eq1.290}) is solvable for all $\sigma>1$.

\begin{theorem}
\label{th2}
Let $\Omega$ be a bounded smooth domain in $\bfR^n$
and $\sigma>1$. Suppose $f$ satisfies
 conditions (\ref{eq1.85})-(\ref{eq1.150}) with $K = K_n^+$. Then for any
$\epsilon > 0$ sufficiently small, there exists an admissible
solution $u^{\epsilon} \in C^\infty (\bar{\Omega})$ of the Dirichlet
problem (\ref{eq1.250}),(\ref{eq1.290}) . Moreover, $u^{\epsilon}$
satisfies the {\em a priori} estimates
\begin{equation}
\label{eq1.300}
\sqrt{1 - |D u^{\epsilon}|^2} = \frac{1}{\sigma} + O(\epsilon)
\;\;\; \mbox{on $\partial \Omega$}
\end{equation}
and
\begin{equation}
\label{eq1.310}
u^{\epsilon}|D^2 u^{\epsilon}| \leq C
\;\;\; \mbox{in $\Omega$}
\end{equation}
where $C$ is independent of $\epsilon$.
\end{theorem}

In the proof of Theorem \ref{th2} we mostly follow the method of \cite{GSS}, \cite{GS11} except that we will appeal to the 
duality results of Section \ref{sec1.5} to use the global maximal principle of \cite{GS11} to control the principal curvatures and prove the first inequality in \eqref{eq1.280}. Because there is no a priori uniqueness (i.e $G_u \geq 0$ need not hold)  it is not sufficient to derive  estimates just for solutions of (\ref{eq1.250}) with constant right hand side $\sigma$,  one must also consider perturbations. {\em However to avoid undue length and tedious repetition, we will prove the estimates for solutions of constant curvature $\sigma$ as the necessary modifications are straightforward} (see \cite{GSS}).\\

By Theorem \ref{th2}, the hyperbolic  principal curvatures of the
admissible solution $u^{\epsilon}$ of the Dirichlet
problem (\ref{eq1.250}),(\ref{eq1.290}) are
uniformly bounded above independent of $\epsilon$.
Since $f(\kappa[u^{\epsilon}])=\sigma$ and $f=0$ on $\partial K_n^+$,
the hyperbolic principal curvatures admit a uniform positive lower bound
independent of $\epsilon$ and therefore (\ref{eq1.250}) is uniformly
elliptic on compact subsets of $\Omega$ for the solution $u^{\epsilon}$.
By the interior estimates of Evans and Krylov, we obtain uniform
$C^{2,\alpha}$ estimates for any compact subdomain of $\Omega$.
The proof of Theorem \ref{th1} is now routine. \\

An outline of the contents of the paper are as follows. Section \ref{sec1.5} contains the important duality results, 
Theorem \ref{th2.1} and Corollary \ref{cor2.1}. Section \ref{sec2} contains preliminary formulas and computations that are used in Section \ref{sec3} to prove the asymptotic angle result Theorem \ref{th3.1}. In Section \ref{sec4} we use
the linearized operator to bound the principal curvatures of a solution on the boundary. Here is where the condition 
\eqref{eq1.150} comes into play. Finally in Section \ref{sec6} we use duality to establish global curvature bounds and complete the proof of Theorem \ref{th2}. The use of duality to prove this global estimate is unusual but seems to be necessary since $F(A)$ is a concave function of A but G is a convex function of $\{u_{ij}\}$.

\section{The Gauss map and Legendre transform}
\label{sec1.5}
\setcounter{equation}{0}

Let $\psi: \Sigma^n \goto \mathcal{H}^{n+1}$ be a strictly locally convex spacelike immersion  with prescribed (compact)  boundary $\Gamma$ either in the timeslice $L_{\tau}$ or in $L_{\infty}=\partial_{\infty} \mathcal{H}$. We are constructing such $\Sigma^n$ as the graph of an admissible
function u :
\[ S=\{(x,u(x)) \in \bfR^{n+1}_+ : u \in C^{\infty}(\ol{\Omega}), u(x)>0,\, \, |\nabla u(x)| <1\}~,\]
where $\Omega$ is a smooth bounded domain in $\bfR^n$. We know that the Gauss map 
\[\vn: \Sigma^n \goto \bfH^{n+1}\]
 takes values in hyperbolic space.  Using the map $\phi$  defined in \eqref{eq1.50} that was used to identify the de Sitter and upper halfspace models of the steady state space $\mathcal{H}^{n+1}$,
 Montiel \cite{Montiel1} showed that if we use the  upper halfspace representation for both $\mathcal{H}$ and $\bfH^{n+1}$, then the Gauss map $\vn$ corresponds to the map 
 \[ L : S \goto \bfH^{n+1}\]
 defined by
 \be \label{eq2.10}
  L((x,u(x))=(x-u(x)\nabla u(x), u(x)\sqrt{1-|\nabla u|^2}) \hspace{.1in} x\in \Omega ~.
  \ee
 We now identify the map $L$ in terms of a hodograph and associated Legendre transform. Define the map $y=\nabla p(x): \Omega \subset \bfR^n \goto \bfR^n$ by
 \be \label {eq2.30}
 y =\nabla p(x),\, \, x\in \Omega \hspace{.1in} \mbox{where $p(x)= \frac12(x^2-u(x)^2)$}.
\ee
Note that $p$ is strictly convex in the Euclidean sense by \eqref{eq1.245} and hence the map $y$ is globally one to one. Therefore $v(y):=u(x)\sqrt{1-|\nabla u(x)|^2}$ is well defined in 
$\Omega^*:= y(\Omega)$. The associated Legendre transform is the function $q(y)$ defined in  $\Omega^*$
  by $ p(x)+q(y)= x \cdot y$ or $q(y)=-p(x)+x\cdot \nabla p(x)$.

\begin{lemma}\label{lemma2.1} The Legendre transform $q(y)$ is given by
\[q(y)=\frac12(y^2 +v(y)^2)\hspace{.1in} \mbox{where $v(y):=u(x)\sqrt{1-|\nabla u(x)|^2}$}~.\]
Moreover, $\sqrt{1+|\nabla v(y)|^2}=(1-|\nabla u|^2)^{-\frac12}$ and $u(x)=v(y)\sqrt{1+|\nabla v(y)|^2}$. Therefore $x=\nabla q(y),\, (q_{ij}(y))=(p_{ij}(x))^{-1}$ and the inverse map $L^*$ of L is given by $L^*(y, v(y))=(x,u(x))$.
\end{lemma}

\begin{proof} We calculate 
\begin{eqnarray*}
&p(x)+q(y)= \frac12(x^2-u(x)^2)+ \frac12(y^2 +v(y)^2)\\
&=\frac12(x^2-u(x)^2)+\frac12(x^2-2u(x) x\cdot \nabla u(x)+
u^2 |\nabla u|^2)+\frac12(u^2(1-|\nabla u(x)|^2)\\
&=x^2-u(x) x\cdot \nabla u(x)=x\cdot y~,
\end{eqnarray*}
as required. It is then standard that $x=\nabla q(y)$ and $(q_{ij}(y))=(p_{ij}(x))^{-1}$.
Then $x=\nabla q(y)=y+v\nabla v(y)$ and $y=x-u(x)\nabla u(x)$ implies $u\nabla u=v\nabla v$ so
$u^2 |\nabla u|^2 =v^2 |\nabla v|^2=u^2(1-|\nabla u|^2) |\nabla v|^2$ and so $|\nabla v(y)|^2=
\frac{\nabla u(x)|^2}{1-|\nabla u(x)|^2}$. Therefore,
\[\sqrt{1+|\nabla v(y)|^2}=(1-|\nabla u|^2)^{-\frac12} \hspace{.1in}\mbox{ and $u(x)=v(y)\sqrt{1+|\nabla v(y)|^2}$}~. \]
\end{proof}

\begin{theorem} \label{th2.1} Let L be defined by \eqref{eq2.10} and let y be defined by \eqref{eq2.30}. Then the image of S by L is the locally strictly convex graph (with respect to the induced hyperbolic metric)
\[ S^*=\{(y, v(y)\in \bfR^{n+1}_+ : u^* \in C^{\infty}(\ol{\Omega^*}), u^*(y)>0,\]
 with principal curvatures
$\kappa_i^*=(\kappa_i)^{-1}$ . Here $\kappa_i>0,\, i=1,\ldots,n$ are the principal curvatures of
S with respect to the induced de Sitter metric. Moreover the inverse map $L^*: S^* \goto S$
defined by
\[L^*((y,v(y))=(y+v(y)\nabla v(y), v(y)\sqrt{1+|\nabla v(y)|^2}\hspace{.1in}y \in \Omega^*\]
is the dual Legendre transform and hodograph map $x=\nabla q(y)$.
\end{theorem}

\begin{proof} By Lemma \ref{lemma2.1} it remains only to show $\kappa_i^*=(\kappa_i)^{-1}$ . The principal curvatures
of $S, \,S^*$ are respectively the eigenvalues of the matrices
\[A[u]=(\gamma^{ij})(h_{ij})(\gamma^{ij}), \hspace{.1in} A[v]=({\gamma^*}^{ij}))(h_{ij}^*)({\gamma^*}^{ij})~,\]
where 
\[g^*_{ij}=\frac{\delta_{ij}+v_i v_j}{v^2},\hspace{.1in} ({\gamma^*}^{ij})=(g^*_{ij})^{-\frac12}, \hspace{.1in} h_{ij}^*=\frac{\delta_{ij}+vv_{ij}+v_i v_j}{v^2\sqrt{1+|\nabla v|^2}}~.\]
By Lemma \ref{lemma2.1},
\[h_{ij}^*=\frac{q_{ij}}{v^2\sqrt{1+|\nabla v|^2}}=\frac{u^2 \sqrt{1-|\nabla u|^2}}{v^2 u^2}q_{ij}=\frac1{u^2 v^2}(h_{ij})^{-1}~,\]
\[g^*_{ij}=\frac{\delta_{ij}+v_i v_j}{v^2}=\frac{\delta_{ij}+\frac{u_i u_j}{1-|\nabla u|^2}}{v^2}=\frac{g^{ij}}{u^2 v^2},\hspace{.1in}
({\gamma^*}^{ij})=uv (\gamma^{ij})^{-1}~,\]
and therefore
$A[v]=(A[u])^{-1}$ completing the proof.
\end{proof}

\begin{corollary} \label{cor2.1} If the graph $S=\{(x,u(x): x\in \Omega\} $ 
is a strictly locally convex spacelike graph with constant curvature $f(\kappa)=\sigma>1$ in the steady state space $\mathcal{H}^{n+1}$ with $\partial_{\infty}\mathcal{H}^{n+1}=\Gamma=\Omega$, then then dual graph  
$S^*=\{(x-u(x)\nabla u(x), u(x) \sqrt{1+|\nabla u|^2})=(y,v(y): y \in \Omega^*\}$
 is a strictly locally convex graph with principal curvatures $\kappa_i^*=(\kappa_i)^{-1}$ 
 of constant curvature $f^*(\kappa) =\sigma^{-1}$ in $\bfH^{n+1}$ with $\partial_{\infty}\bfH^{n+1}=\Gamma=\partial \Omega$ and conversely.
 \end{corollary}

\section{Formulas on hypersurfaces}
\label{sec2}
\setcounter{equation}{0}

In this section we will derive some basic identities on a hypersurface
by comparing the induced metric in steady state space $\mathcal{H}^{n+1}\subset dS_{n+1}$ and Minkowski space.\\

Let $\Sigma$ be a hypersurface in $\mathcal{H}^{n+1}$. We shall use $g$ and $\nabla$
to denote the induced metric and Levi-Civita connection
on $\Sigma$, respectively. As $\Sigma$ is also a submanifold of $\bfR^{n,1}$,
we shall usually distinguish a geometric quantity with respect to the
Minkowski metric by adding a `tilde' over the corresponding
quantity.
For instance, $\tg$ denotes the induced  metric on $\Sigma$
from $\bfR^{n,1}$, and $\tnabla$ is its Levi-Civita connection.

Let $\vx$ be the position vector of $\Sigma$ in $\bfR^{n,1}$ and set
\[ u = \vx \cdot \bf{e} \]
where $\bf{e}$$=(0,\cdots, 0, 1)$ is the unit vector in the positive
$x_{n+1}$ direction in $\bfR^{n+1}$, and `$\cdot$' denotes the Euclidean
inner product in $\bfR^{n+1}$.
We refer $u$ as the {\em height function} of $\Sigma$.

Throughout the paper we assume $\Sigma$ is orientable and let ${\bf n}$ be
a (global)
unit normal vector field to $\Sigma$ with respect to the de Sitter metric.
This also determines a unit normal $\nu$ to $\Sigma$ with respect to the
Minkowski metric by the relation
\[ \nu = \frac{{\bf n}}{u}. \]
We denote $ \nu^{n+1} = \bf{e} \cdot \nu$.

Let $(z_1, \ldots, z_n)$ be local coordinates and
\[ \tau_i = \frac{\partial}{\partial z_i}, \;\; i = 1, \ldots, n. \]
The de Sitter and Minkowski metrics of $\Sigma$ are given by
\[ g_{ij} = \langle \tau_i, \tau_j \rangle_D, \;\;
   \tg_{ij} = \langle\tau_i, \tau_j\rangle_M = u^2 g_{ij}, \]
while the second fundamental forms are
\begin{equation}
\label{eq3.10}
 \begin{aligned}
   h_{ij} \,& = \langle D_{\tau_i} \tau_j, {\bf n} \rangle_D
              = - \langle D_{\tau_i} {\bf n}, \tau_j \rangle_D, \\
\thh_{ij} \,& =\langle\nu, \tD_{\tau_i} \tau_j\rangle_M
              = - \langle\tau_j, \tD_{\tau_i} \nu\rangle_M,
    \end{aligned}
\end{equation}
where $D$ and $\tD$ denote the Levi-Civita connection of $\mathcal{H}^{n+1}$
and $\bfR^{n,1}$, respectively, and $\langle\; ,\; \rangle_D,$ $\langle\; , \;\rangle_M$ denote the corresponding inner product.\\

The Christoffel symbols $\Gamma_{ij}^k$ and $ \tilde{\Gamma}_{ij}^k$ are related by the formula
\begin{equation}
\label{eq2.70}
\Gamma_{ij}^k = \tilde{\Gamma}_{ij}^k - \frac{1}{u}
   (u_i \delta_{kj} + u_j \delta_{ik} - \tg^{kl} u_l \tg_{ij}).
\end{equation}
It follows that for $v \in C^2 (\Sigma)$
\begin{equation}
\label{eq2.80}
\nabla_{ij} v = v_{ij} - \Gamma_{ij}^k v_k
  = \tilde{\nabla}_{ij} v + \frac{1}{u}
    (u_i v_j + u_j v_i - \tg^{kl} u_k v_l \tg_{ij})
\end{equation}
where (and in sequel)
\[ v_i = \frac{\partial v}{\partial z_i}, \;
   v_{ij} = \frac{\partial^2 v}{\partial z_i z_j}, \; \mbox{etc.} \]
In particular,
\begin{equation}
\label{eq2.90}
\begin{aligned}
\nabla_{ij} u
 \,& = \tilde{\nabla}_{ij} u + \frac{2 u_i u_j}{u}
       - \frac{1}{u} \tg^{kl} u_k u_l \tg_{ij}
  \end{aligned}
\end{equation}
and
\begin{equation}
\label{eq2.100}
\nabla_{ij} \frac{1}{u}
  = - \frac{1}{u^2} \tilde{\nabla}_{ij} u
  + \frac{1}{u^3} \tg^{kl} u_k u_l \tg_{ij}.
\end{equation}
Moreover,
\begin{equation}
\label{eq2.110}
\begin{aligned}
\nabla_{ij} \frac{v}{u}
 \,&  = v \nabla_{ij} \frac{1}{u}
        + \frac{1}{u} \tilde{\nabla}_{ij} v
        - \frac{1}{u^2} \tg^{kl} u_k v_l \tg_{ij}.
 \end{aligned}
\end{equation}

In $\bfR^{n,1}$,
\begin{equation}
\label{eq2.120}
\begin{aligned}
             \tg^{kl} u_k u_l
      \,& = |\tilde{\nabla} u|^2 = (\nu^{n+1})^2-1 \\
            \tilde{\nabla}_{ij} u
      \,& = -\thh_{ij} \nu^{n+1}.
 \end{aligned}
\end{equation}
Therefore, by \eqref{eq1.230} and \eqref{eq2.100},
\begin{equation}
\label{eq2.130}
\begin{aligned}
\nabla_{ij} \frac{1}{u}
 \,&  =-\frac{1}{u^2}\tnabla_{ij}u+\frac{1}{u^3}\tg_{ij}\left[(\nu^{n+1})^2-1\right]\\
 \,&  =\frac{1}{u}\left(h_{ij}\nu^{n+1}-g_{ij}\right) .
 \end{aligned}
\end{equation}
We note that \eqref{eq2.110} and \eqref{eq2.130} still hold for
general local frames $\tau_1, \ldots, \tau_n$.
In particular, if $\tau_1, \ldots, \tau_n$ are orthonormal in the de Sitter
metric, then
$g_{ij} = \delta_{ij}$ and $\tg_{ij} = u^2 \delta_{ij}$.

We now consider equation~\eqref{eq1.10} on $\Sigma$.
Let $\mathcal{A}$ be the vector space of $n \times n$ matrices and
\[ \mathcal{A}^+= \{A = \{a_{ij}\} \in \mathcal{A}: \lambda (A) \in K_n^+\}, \]
where $\lambda (A) = (\lambda_1, \dots, \lambda_n)$ denotes the eigenvalues of $A$.
Let $F$ be the function defined by
\begin{equation}
\label{eq2.140}
F (A) = f (\lambda (A)), \;\; A \in \mathcal{A}^+
\end{equation}
and denote
\begin{equation}
\label{eq2.150}
F^{ij} (A) = \frac{\partial F}{\partial a_{ij}} (A), \;\;
  F^{ij, kl} (A) = \frac{\partial^2 F}{\partial a_{ij} \partial a_{kl}} (A).
\end{equation}
Since $F (A)$ depends only on the eigenvalues of $A$, if $A$ is symmetric
then so is the matrix $\{F^{ij} (A)\}$. Moreover,
\[ F^{ij} (A) = f_i \delta_{ij} \]
when $A$ is diagonal, and
\begin{equation}
\label{eq2.160}
 F^{ij} (A) a_{ij} = \sum f_i (\lambda (A)) \lambda_i = F (A),
\end{equation}
\begin{equation}
\label{eq2.170}
F^{ij} (A) a_{ik} a_{jk} = \sum f_i (\lambda (A)) \lambda_i^2.
\end{equation}

Equation~\eqref{eq1.10} can therefore be rewritten in a  local
frame $\tau_1, \ldots, \tau_n$ in the form
\begin{equation}
\label{eq2.180}
F (A[\Sigma]) = \sigma
\end{equation}
where $A[\Sigma] = \{g^{ik} h_{kj}\}$.
Let $F^{ij} = F^{ij} (A[\Sigma])$, $F^{ij, kl} = F^{ij, kl} (A[\Sigma])$.

\begin{lemma}
\label{lem2.10}
Let $\Sigma$ be a smooth hypersurface in $\mathcal{H}^{n+1}$ satisfying
equation~\eqref{eq1.10}. Then in a local orthonormal frame,
\begin{equation}
\label{eq2.190}
  F^{ij} \nabla_{ij} \frac{1}{u}
    = \frac{\sigma \nu^{n+1}}{u}-\frac{1}{u} \sum f_i.
\end{equation}
and
\begin{equation}
\label{eq2.200}
  F^{ij} \nabla_{ij} \frac{\nu^{n+1}}{u} =
-\frac{\sigma}{u} + \frac{\nu^{n+1}}{u} \sum f_i \kappa_i^2.
\end{equation}
\end{lemma}

\begin{proof}
The first identity follows immediately from \eqref{eq2.130},
\eqref{eq2.160} and assumption \eqref{eq1.140}. To prove
\eqref{eq2.200} we recall the identities in $\bfR^{n,1}$
\begin{equation}
\label{eq2.210}
 \begin{aligned}
        (\nu^{n+1})_i
  \,& = - \thh_{il} \tg^{lk} u_k, \\
        \tilde{\nabla}_{ij} \nu^{n+1}
  \,& = - \tg^{kl} (-\nu^{n+1} \thh_{il} \thh_{kj} + u_l \tnabla_k \thh_{ij}).
\end{aligned}
\end{equation}
By \eqref{eq2.150}, \eqref{eq2.160}, \eqref{eq2.170}, and
$\tg^{ik} = \delta_{jk}/u^2$ we see that
\begin{equation}
\label{eq2.220}
 \begin{aligned}
 F^{ij} \tg^{kl} \thh_{il} \thh_{kj}
   = \,& \frac{1}{u^2} F^{ij} \thh_{ik} \thh_{kj} \\
   = \,& F^{ij} (h_{ik} h_{kj} - 2 \nu^{n+1} h_{ij} + (\nu^{n+1})^2 \delta_{ij}) \\
   = \,& f_i \kappa_i^2 - 2 \nu^{n+1} \sigma + (\nu^{n+1})^2 \sum f_i.
\end{aligned}
\end{equation}

As a hypersurface in $\bfR^{n,1}$,
it follows from \eqref{eq1.240} that $\Sigma$ satisfies
\[ f (u \tkappa_1 + \nu^{n+1}, \ldots, u \tkappa_n + \nu^{n+1}) = \sigma, \]
or equivalently,
\begin{equation}
\label{eq2.230}
F (\{\tg^{il} (u \thh_{lj} + \nu^{n+1} \tg_{lj})\}) = \sigma.
\end{equation}
Differentiating equation~\eqref{eq2.230}
and using $\tg_{ij} = u^2 \delta_{ij}$, $\tg^{ik} = \delta_{ik}/u^2$, we obtain
\begin{equation}
\label{eq2.240}
 F^{ij} (u^{-1} \tnabla_k \thh_{ij} + u^{-2}u_k \thh_{ij}
     + (\nu^{n+1})_k \delta_{ij}) = 0.
\end{equation}
That is,
\begin{equation}
\label{eq2.250}
\begin{aligned}
F^{ij} \tnabla_k \thh_{ij} + \,& u(\nu^{n+1})_k \sum F^{ii}
   = - \frac{u_k}{u} F^{ij} \thh_{ij}  \\
   = \,& - u_k F^{ij} (h_{ij} - \nu^{n+1} \delta_{ij}) \\
   = \,& - u_k \Big(\sigma - \nu^{n+1} \sum f_i\Big).
  \end{aligned}
\end{equation}

Finally, combining \eqref{eq2.110}, \eqref{eq2.190}, \eqref{eq2.210},
\eqref{eq2.220}, \eqref{eq2.250}, and the first identity in \eqref{eq2.120},
 we derive
\begin{equation}
\label{eq2.260}
\begin{aligned}
  F^{ij} \nabla_{ij} \frac{\nu^{n+1}}{u}
  = \,& \nu^{n+1} F^{ij} \nabla_{ij} \frac{1}{u}
        + \frac{|\tnabla u|^2}{u} F^{ij} \thh_{ij}
        - \frac{\nu^{n+1}}{u^3} F^{ij} \thh_{ik} \thh_{kj} \\
  = \,& \frac{\nu^{n+1}}{u} \Big(\nu^{n+1}\sigma-\sum f_i\Big)
        + \frac{|\tnabla u|^2}{u} \Big(\sigma - \nu^{n+1} \sum f_i\Big)  \\
    \,& +\frac{\nu^{n+1}}{u} \Big(f_i \kappa_i^2
     - 2 \nu^{n+1} \sigma + (\nu^{n+1})^2 \sum f_i\Big) \\
  = \,& -\frac{\sigma}{u} + \frac{\nu^{n+1}}{u} \sum f_i \kappa_i^2.
  \end{aligned}
\end{equation}
This proves \eqref{eq2.200}.
\end{proof}

\section{The asymptotic angle maximum principle and gradient estimates}
\label{sec3}
\setcounter{equation}{0}

In this section we show that the upward unit normal of a solution tends to
a fixed asymptotic angle on approach to the asymptotic boundary and that this holds approximately for the solutions of the approximate problem.. This implies a global (spacelike)
gradient bound on solutions.\\

Our estimates are all based on the use of special barriers (see  section 3 of \cite{Montiel1}). These correspond to horospheres for the dual problem in hyperbolic space and our argument follows that of  section 3 of \cite{GS08}. Let
 \[Q(r, c)=\{x\in\bfR^{n,1} \mid \langle x-c, x-c\rangle_M \leq -r^2\}\]
 be a ball of radius $r$ centered at $c$ in Minkowski space, where $c\in\bfR^{n+1}.$ Moreover, let $Q_+(r, c)$ denote the region above the upper hyperboloid
 and $Q_-(r,c)$ denote the region below the lower hyperboloid.  If we choose $a=(a', -r\sigma),$
 then $S_+(r,a)=\partial Q_+(r, a)\cap \mathcal{H}^{n+1}$ is an umbilical hypersurface in $\mathcal{H}^{n+1}$ with constant curvature $\sigma$ with respect to its upward normal vector. For convenience we sometimes call $S_+(r, a)$ an upper hyperboloid
 of constant curvature $\sigma$ in $\mathcal{H}^{n+1}.$ Similarly,
 when we choose $b=(b', r\sigma),$ then $S_-(r, b)=\partial Q_-(r, b)\cap\mathcal{H}^{n+1}$ is the lower hyperboloid
 of constant curvature $\sigma$ with respect to its upward normal vector.
 These hyperboloids serve as useful barriers.\\

 Now let $\Sigma$ be a hypersurface in $\mathcal{H}^{n+1}$ with $\partial\Sigma\subset P(\e):=\{x_{n+1}=\e\}$
 so $\Sigma$ separates $\{x_{n+1}\geq\e\}$ into an inside (bounded) region and outside (unbounded) one.
 Let $\Omega$ be the region in $\bfR^n\times\{0\}$ such that its vertical lift $\Omega^\e$ to $P(\e)$ is bounded by
 $\partial\Sigma$ (and $\bfR^n\setminus\Omega$ is connected and unbounded). (It is allowable that $\Omega$ have several connected
 components.)  Suppose $\kappa[\Sigma]\in K_n^+$ and $f(\kappa)=\sigma\in(1, \infty)$ with respect to its outer normal.

 \begin{lemma}\label{lem3.0}
 $$\begin{aligned}
 (i)\;&\Sigma\cap\{x_{n+1}<\epsilon\}=\emptyset.\\
 (ii)\;&\mbox{If $\partial\Sigma\subset Q_-(r, b)$, then $\Sigma\subset Q_-(r, b).$}\\
 (iii)\;&\mbox{If $Q_-(r, b)\cap P(\e)\subset\Omega^\e,$ then $Q_-(r, b)\cap\Sigma=\emptyset.$}\\
 (iv)\;&\mbox{If $Q_+(r, a)\cap\Omega^{\e}=\emptyset,$ then $Q_+(r, a)\cap\Sigma=\emptyset.$}\\
 \end{aligned}$$
 \end{lemma}
\begin{proof}
For (i) let $c=\min_{x\in\Sigma}x_{n+1}$ and suppose $0<c<\e.$ Then the horizontal plane $P(c)$ satisfies $f(\kappa)=1$
with respect to the upward normal, lies below $\Sigma,$ and has an interior contact point. Then $f(\kappa[\Sigma])\leq 1$
at this point, which leads to a contradiction (notice that in the Euclidean case we have the reverse inequality). 

For (ii), (iii), (iv) we consider the family $\{h_s\}_{s\in R}$ of isometries of $\mathcal{H}^{n+1}$ 
consisting of Euclidean homotheties. We perform
homothetic dilations from $(a',0)$ and $(b', 0)$ respectively, and then use the maximum principle.
For (ii), choose $s_0$ big enough so that $h_{s_0}(Q_-(r, b))$ contains $\Sigma$ and then decrease $s.$ Since the curvature of
$\Sigma$ and $S_-(r, b)$ are calculated with respect to their outward normals and both hypersurfaces satisfy $f(\kappa)=\sigma,$
there cannot be a first contact. 

For (iii) and (iv) we shrink $Q_+(r, a)$ and $Q_-(r, b)$ until they are respectively inside and outside
$\Sigma.$ When we expand $Q_-(r, b)$ there cannot be a first contact as above. Now decrease $s$ to a certain value $s_1\in R$
such that $h_{s_1}(Q_+(r, a))$ is disjoint from $\Sigma$ (outside of). Then we increase $s_1$ and suppose there is a first interior contact.
The outward normal to $\Sigma$ at this contact point is the upward normal to $S_+(r, a).$ Since the curvatures of $S_+(r, a)$
are calculated with respect to the upward normal and $S_+(r, a)$ satisfies $f(\kappa)=\sigma,$ 
we have a contradiction of the maximum principle.
\end{proof}

\begin{theorem}
\label{th3.1}
Let $\Sigma$ be a smooth strictly locally convex spacelike hypersurface in $\mathcal{H}^{n+1}$
satisfying equation~\eqref{eq1.10}. Suppose $\Sigma$ is globally a graph:
\[ \Sigma = \{(x, u (x)): x \in \Omega\} \]
where $\Omega$ is a domain in $\bfR^n \equiv \partial \mathcal{H}^{n+1}$. Then
\be
\label{eq3.20}
F^{ij}\nabla_{ij}\frac{\sigma - \nu^{n+1}}{u}
   = \frac{\sigma}{u}\left(1-\sum f_i\right)+\frac{\nu^{n+1}}{u}\left(\sigma^2-\sum f_i\kappa_i^2\right)  \leq 0
\ee
and so,
\begin{equation}
\label{eq3.30}
\frac{\sigma - \nu^{n+1}}{u}
  \geq \inf_{\partial \Sigma} \frac{\sigma - \nu^{n+1}}{u}
\;\; \mbox{on $\Sigma$}.
\end{equation}
Moreover, if $u = \epsilon>0$ on  $\partial \Omega$, then there exists
$\epsilon_0 > 0$ depending only on $\partial \Omega$, such that for all
$\epsilon \leq \epsilon_0$,
\begin{equation}
\label{eq3.40}
\frac{r_1\sqrt{\sigma^2-1}}{r_1^2-\e^2}+\frac{\e(\sigma-1)}{r_1^2-\e^2}>\frac{\sigma - \nu^{n+1}}{u}
   > -\frac{r_2\sqrt{\sigma^2-1}}{r_2^2-\e^2}-\frac{\e(1+\sigma)}{r_2^2-\e^2}
\;\; \mbox{on $\partial\Sigma$}
\end{equation}
where $r_2,$ $r_1$ are the maximal radius of exterior and interior spheres to
$\partial \Omega$, respectively. In particular, $\nu^{n+1}\rightarrow\sigma$
on $\partial\Sigma$ as $\e\rightarrow 0.$
\end{theorem}

\begin{proof}
It's easy to see that \eqref{eq3.20} follows from equations \eqref{eq2.190}, \eqref{eq2.200} and \eqref{eq1.170}, \eqref{eq1.180} .
Thus, \eqref{eq3.30} follows from the maximum principle.

In order to prove \eqref{eq3.40}, we first assume $r_2<\infty.$ Fix a point $x_0\in\partial\Omega$ 
and let $e_1$ be the outward pointing unit normal to $\partial\Omega$
at $x_0.$ Let $S_+(R_2, a),$ $S_-(R_1, b)$ be the upper and lower hyperboloid with 
centers $a=(x_0+r_2e_1, -R_2\sigma),$ $b=(x_0-r_1e_1, R_1\sigma)$
and radii $R_2$, $R_1$ respectively satisfying
\[r_2^2-(R_2\sigma+\e)^2=-R_2^2,\;\;\;r_1^2-(R_1\sigma-\e)^2=-R_1^2.\]
Then $Q_-(R_1, b)\cap P(\e)$ is an n-ball of radius $r_1$ internally tangent to $\partial\Omega^\e$ at $x_0$
while $Q_+(R_2, a)\cap P(\epsilon)$ is an n-ball of radius $r_2$ externally tangent to $\partial\Omega^\e$ at $x_0.$
By Lemma \ref{lem3.0} (iii) and (iv), $Q_\pm\cap\Sigma=\emptyset.$ Hence
\be\label{eq3.01}
\frac{\sigma R_1-u}{R_1}<\nu^{n+1}<\frac{\sigma R_2+u}{R_2}.
\ee
Moreover, by a simple calculation we have
\be\label{eq3.02}
\frac{1}{R_1}=\frac{-\e\sigma+\sqrt{r_1^2(\sigma^2-1)+\e^2}}{r_1^2-\e^2}
<\frac{r_1\sqrt{\sigma^2-1}}{r_1^2-\e^2}+\frac{\e(\sigma-1)}{r_1^2-\e^2},
\ee
\be\label{eq3.03}
\frac{1}{R_2}=\frac{\e\sigma+\sqrt{(\sigma^2-1)r_2^2+\e^2}}{r_2^2-\e^2}
<\frac{r_2\sqrt{\sigma^2-1}}{r_2^2-\e^2}+\frac{\e(1+\sigma)}{r_2^2-\e^2}.
\ee

Finally \eqref{eq3.40} follows from \eqref{eq3.01}, \eqref{eq3.02} and \eqref{eq3.03}.

If $r_2=\infty$, in the above argument one can replace $r_2$ by any $r>0$ and then let $r\rightarrow\infty.$
\end{proof}

From Theorem~\ref{th3.1} we conclude
\begin{corollary}
\label{cor3.2}
Let $\Omega$ be a bounded smooth domain in $\bfR^n$
and $\sigma >1$. Suppose $f$ satisfies
 conditions (\ref{eq1.85})-(\ref{eq1.150}) with $K = K_n^+$. Then for any
$\epsilon > 0$ sufficiently small, any admissible
solution $u^{\epsilon} \in C^\infty (\bar{\Omega})$ of the Dirichlet
problem  (\ref{eq1.250}),(\ref{eq1.290}) satisfies the {\em apriori} estimate
\be
\label{eq3.50}
|\nabla u^{\epsilon}| \leq C< 1
\;\;\; \mbox{in $\Omega$}
\ee
where $C$ is independent of $\epsilon$.
\end{corollary}

\section{The linearized operator and boundary estimates for second derivatives.}
\label{sec4}
\setcounter{equation}{0}
In this section we establish boundary estimates for second derivatives of admissible solutions.

\begin{theorem}
\label{th4.0}
Suppose that $f$ satisfies conditions (\ref{eq1.85})-(\ref{eq1.150}) with $K = K_n^+$.  If $\e$ is sufficiently small, then
\be\label{eq4.10}
u|D^2u|\leq C\;\;\;\mbox{on $\partial{\Omega}$}
\ee
where $C$ is independent of $\e.$
\end{theorem}

Define the linearized operator of G at u (recall \eqref{eq1.249})
\be\label{eq4.120}
\mathcal{L}=G^{st}\partial_s\partial_t+G^s\partial_s+G_u
\ee
 where
\be\label{eq4.130}
G^{st}=\frac{\partial G}{\partial u_{st}},\;\;G^s=\frac{\partial G}{\partial u_s},\;\;G_u=\frac{\partial G}{\partial u}.
\ee
Note that
\be\label{eq4.140}
G^{st}=-\frac{u}{w}F^{ij}\gamma^{is}\gamma^{jt},\;\;G^{st}u_{st}=uG_u=\sigma-\frac{\sum f_i}{w}.
\ee
After some straightforward but tedious calculations we derive
\be\label{eq4.150}
G^s=\frac{u_s}{w^2}\sigma+2\frac{F^{ij}a_{ik}}{w(1+w)}
\left(u_k\gamma^{sj}w+u_j\gamma^{ks}\right)-2\frac{F^{ij}u_i\gamma^{sj}}{w^2}.
\ee
It follows that
\begin{lemma}
\label{lem4.10}
Suppose that $f$ satisfies \eqref{eq1.50}, \eqref{eq1.60}, \eqref{eq1.80} and \eqref{eq1.90}. Then
\be\label{eq4.160}
|G^s|\leq C_0(1+\sum f_i),
\ee
where $C_0$ denotes a controlled constant independent of $\e.$
\end{lemma}
Since $\gamma^{sj}u_s=u_j/w,$
\be\label{eq4.170}
G^su_s=\frac{1-w^2}{w^2}\sigma+2\frac{F^{ij}a_{ik}u_ku_j}{w^2}-2\frac{F^{ij}u_iu_j}{w^3}.
\ee
Let
\be\label{eq4.180}
\mathcal{L'}=-\mathcal{L}+G_u=-G^{st}\partial_s\partial_t-G^s\partial_s.
\ee
Then from \eqref{eq4.140} and \eqref{eq4.170} we obtain
\be\label{eq4.190}
\begin{aligned}
\mathcal{L'}u=\;&\frac{1}{w}\sum f_i-\frac{\sigma}{w^2}-2\frac{F^{ij}a_{ik}u_ku_j}{w^2}+2\frac{F^{ij}u_iu_j}{w^3}\\
          \leq\;& C_1+C_2\sum f_i.
\end{aligned}
\ee
In the following we denote by  $C_1, C_2, \ldots$ controlled constants independent of $\e.$\\

We will employ a barrier function of the form
\be\label{eq4.200}
v=u-\e+td-Nd^2
\ee
where $d$ is the distance function from $\partial\Omega,$ and $t,$ $N$ are positive constants to be determined.
We may take $\delta>0$ small enough so that $d$ is smooth in $\Omega_\delta=\Omega\cap B_{\delta}(0).$

\begin{lemma}
\label{lem4.20}
For $\delta=c_0 \e,\, N=\frac{C_4}{\e},\,t=c_0C_4$ with $C_4$ sufficiently large and $c_0$ sufficiently small independent of $\e$,
\[\mathcal{L'}v\leq-(1+\sum f_i)\;\;\mbox{in $\Omega_\delta$,}\;\;\;v\geq 0\;\;\mbox{on $\partial\Omega_\delta$}.\]
\end{lemma}
\begin{proof}
Since $|Dd|=1$ and $-CI\leq D^2d\leq CI,$ we have
\be\label{eq4.210}
\begin{aligned}
\left|\mathcal{L'}d\right|=\;&\left|-G^{st}d_{st}-G^sd_s\right|\\
                         \leq\;&C_3(1+\sum f_i).
\end{aligned}
\ee
Furthermore, since $d_n(0)=1,$ $d_{\beta}(0)=0$ for all $\beta<n,$ we have, when $\delta>0$ sufficiently small,
\be\label{eq4.220}
\begin{aligned}
-G^{st}d_sd_t\geq\;& -G^{nn}d_n^2-2\sum_{\beta<n}G^{n\beta}d_nd_\beta\\
             \geq\;&\frac{-1}{2}\sum G^{nn}=\frac{u}{2w}\sum F^{ij}\gamma^{in}\gamma^{jn}\\
             \geq\;&\frac{u}{2nw}\sum F^{ii}.
\end{aligned}
\ee
Therefore,
\be\label{eq4.230}
\begin{aligned}
\mathcal{L'}v=\;&\mathcal{L'}u+(t-2Nd)\mathcal{L'}d-2NG^{st}d_sd_t\\
          \leq\;&C_1+C_2\sum f_i+C_3(t+2N\delta)(1+\sum f_i)-\frac{N\e}{nw}\sum f_i\\
          \leq\;&(C_1+tC_3+2N\delta C_3-\frac{N\e}{2n})+(C_2+tC_3+2N\delta C_3-\frac{N\e}{2n})\sum f_i.\\
\end{aligned}
\ee
Now if we require $0<c_0<\frac{1}{18nC_3}$
and $C_4\gg 3n\max\{C_1, C_2\}+3n.$ let $N=C_4/\e,$ $t=C_4c_0,$ $\delta=c_0\e,$ then
Lemma \ref{lem4.20} is proved.
\end{proof}

The following lemma is proven in \cite{GS08}; it applies to our situation since horizontal rotations are isometries for $\mathcal{H}^{n+1}.$

\begin{lemma}
\label{lem4.30}
Suppose that $f$ satisfies \eqref{eq1.50}, \eqref{eq1.60}, \eqref{eq1.80} and \eqref{eq1.90}. Then
\be\label{eq4.240}
\mathcal{L}(x_iu_j-x_ju_i)=0,\;\;\;\mathcal{L}u_i=0,\;\;1\leq i, j\leq n.
\ee
\end{lemma}

\begin{proof}[Proof of Theorem~\ref{th4.0}] Consider an arbitrary point on $\partial\Omega,$ which we may assume to be the origin of $\bfR^n$ and choose the coordinates so that the positive $x_n$ axis is the interior normal to $\partial\Omega$ at the origin. There exists a uniform constant $r>0$ such that $\partial\Omega\cap B_r(0)$ can be represented as a graph
\[x_n=\rho(x')=\frac{1}{2}\sum_{\alpha, \beta<n}B_{\alpha\beta}x_\alpha x_\beta+O(|x'|^3),\;\;x'=(x_1, \cdots, x_{n-1}).\]
Since $u=\e$ on $\partial\Omega,$ we see that $u(x', \rho(x'))=\e$ and
\[u_{\alpha\beta}(0)=-u_n\rho_{\alpha\beta},\;\;\;\alpha, \beta<n.\]
Consequently,
\[|u_{\alpha\beta}(0)|\leq C|Du(0)|,\;\;\;\alpha, \beta<n,\]
where $C$ depends only on the (Euclidean maximal principal) curvature of $\partial\Omega.$

Next, following \cite{CNS1} we consider for fixed $\alpha<n$ the approximate tangential operator
\be\label{eq4.250}
T=\partial_\alpha+\sum_{\beta<n}B_{\beta\alpha}(x_\beta\partial_n-x_n\partial_\beta).
\ee
We have
\be\label{eq4.260}
\begin{aligned}
|Tu|\leq C,\;\;&\mbox{in $\Omega\cap B_\delta(0)$}\\
|Tu|\leq C|x|^2,\;\;&\mbox{on $\partial\Omega\cap B_\delta(0)$}
\end{aligned}
\ee
since $u=\e$ on $\partial\Omega.$ By Lemma \ref{lem4.30} and \eqref{eq4.140}, \eqref{eq4.260},
\be\label{eq4.270}
\begin{aligned}
\left|\mathcal{L'}(Tu)\right|=\;&\left|-\mathcal{L}Tu+G_u Tu\right|\\
                             =\;&\left|G_u Tu\right|\\
                          \leq\;&\frac{C_5}{\e}(1+\sum f_i).
\end{aligned}
\ee
A straightforward calculation gives
\be\label{eq4.280}
\begin{aligned}
\left|\mathcal{L'}|x|^2\right|=\;&\left|-2\sum G^{ss}-2\sum x_sG^s\right|\\
                           \leq\;& C_6\e(1+\sum f_i).
\end{aligned}
\ee

Now let
\[\Phi=\frac{A}{\e}v+\frac{C}{\delta^2}|x|^2\pm Tu.\]
By Lemma \ref{lem4.20} and \eqref{eq4.270}, \eqref{eq4.280},
\be\label{eq4.290}
\mathcal{L'}\Phi \leq -\frac{A}{\e}(1+\sum f_i)+\frac{C_6C}{{c_0}^2 \e}(1+\sum f_i)+\frac{C_5}{\e}(1+\sum f_i)\;\;\mbox{in $\Omega\cap B_\delta$}
\ee
Choosing $A\gg C_5+\frac{C_6C}{{c_0}^2}$ makes $\mathcal{L'}\Phi\leq 0$ in $\Omega\cap B_\delta.$ It is also easy to see that $\Phi\geq 0$ on $\partial(\Omega\cap B_\delta).$

By the maximum principle $\Phi\geq 0$ in $\Omega\cap B_\delta.$ Since $\Phi(0)=0,$ we have $\Phi_n(0)\geq 0$ which gives
\be\label{eq4.300}
|u_{\alpha n}(0)|\leq \frac{A(u_n(0)+C_4 c_0)}{\e}\leq \frac{C}{\e}.
\ee

Finally to estimate $|u_{nn}(0)|$ we use our hypothesis \eqref{eq1.150} and Theorem \ref{th3.1}. We may assume $[u_{\alpha\beta}(0)],$ $1\leq\alpha, \beta<n,$
to be diagonal. Note that $u_\alpha(0)=0$ for $\alpha<n.$ We have at $x=0$
$$A[u]=\frac{1}{w}
\left[\begin{array}{cccc}
1-uu_{11}&0&\cdots&-\frac{uu_{1n}}{w}\\
0&1-uu_{22}&\cdots&-\frac{uu_{2n}}{w}\\

\vdots&\vdots&\ddots&\vdots\\
-\frac{uu_{n1}}{w}&-\frac{uu_{n2}}{w}&\cdots&1-\frac{uu_{nn}}{w^2}\\
\end{array}\right]. $$

By Lemma 1.2 in \cite{CNS3}, if $|\e u_{nn}(0)|$ is very large, the eigenvalues $\lambda_1, \cdots, \lambda_n$ of $A[u]$
are asymptotically given by
\be\label{eq4.310}
\begin{aligned}
\lambda_\alpha=\;&\frac{1}{w}(1+|\e u_{nn}(0)|)+o(1),\;\;\alpha<n\\
\lambda_n=\;&\frac{|\e u_{nn}(0)|}{w^3}\left(1+O\left(\frac{1}{|\e u_{nn}(0)|}\right)\right).
\end{aligned}
\ee
If $|\e u_{nn}(0)|\geq R$ where $R$ is a controlled constant only depends on $\sigma$. By the hypothesis \eqref{eq1.150} and Theorem \ref{th3.1},
\[\sigma=\frac{1}{w}F(wA[u](0))\geq (\sigma-C\epsilon)F(wA[u](0))\geq (\sigma-C\e)(1+\e_0)\geq \sigma (1+\frac{\e_0}2)\]
leads to a contradiction. Therefore
\[|u_{nn}(0)|\leq\frac{R}{\e}\]
and the proof is complete.
\end{proof}

\section{\texorpdfstring{Completion of the proof of Theorem \ref{th2}}   {Completion of the proof of Theorem 1.5}}
\label{sec6}
\setcounter{equation}{0}

As we emphasized in the introduction, we will derive a global curvature estimate for solutions of the  Dirichlet
problem (\ref{eq1.250}),(\ref{eq1.290}) . In Theorem \ref{th4.0} of the previous section we have shown that the principal
curvatures satisfy $0<\kappa_i \leq C,\, i=1,\ldots,n$ on $\Gamma=\partial \Omega$, hence lie in a compact set E of the cone K.
Since $f(\kappa)=\sigma$ and $f(\kappa) \goto 0$ uniformly on E when any $\kappa_i \goto 0$, it follows that
\be \label{eq6.10}
 \frac1C<\kappa_i \leq C\hspace{.1in} \mbox{ on $\Gamma$ }.
 \ee

We now appeal to  duality. By Corollary \ref{cor2.1}, the dual graph $S^*$ satisfies $f^*(\kappa)=\frac1{\sigma}$
with principal curvatures $\kappa_i^*=(\kappa_i)^{-1}$.  So by \eqref{eq6.10}
\be \label{eq6.30}
 \frac1C<\kappa_i ^* \leq C\hspace{.1in} \mbox{ on $\Gamma^*= L(\Gamma)$ }.
 \ee

Hence by the global maximum principal for principal curvatures proved in Theorem 4.1 of \cite{GS11},
\be \label{eq6.50}
 \frac1{\ol{C}}<\kappa_i ^* \leq \ol{C} \hspace{.1in} \mbox{ on $S^*$ }.
 \ee

Once more using duality to return to the graph S, we obtain the desired global estimate
\be \label{eq6.70}
 \frac1{\ol{C}}<\kappa_i \leq \ol{C} \hspace{.1in} \mbox{ on $S$ }.
 \ee

The proof of Theorem \ref{th1} follows by letting $\e \goto 0$ as mentioned in the introduction.\\

\end{document}